\documentclass[12pt,a4paper]{amsart}
\usepackage[latin1]{inputenc}
\usepackage{latexsym}
\usepackage{color,graphicx,shortvrb}
\usepackage{amsmath, amssymb}
\usepackage{amsfonts}

\newtheorem{theorem}{Theorem}[section]

\newtheorem{corollary}[theorem]{Corollary}

\theoremstyle{definition}

\newtheorem{remark}[theorem]{Remark}

\author{J.~M.~Almira}
\title{A Montel-type theorem for mixed differences}

\begin{document}
\keywords{}



\begin{abstract}
We prove a generalization of classical Montel's theorem for the mixed differences case, for polynomials and exponential polynomial functions, in commutative setting.
\end{abstract}

\subjclass[2010]{Primary 39A05, 39A70.}

\keywords{Equations in iterated differences, polynomials, exponential polynomials, Montel-type theorem}

\maketitle

\markboth{J. M. Almira}{Montel-type theorem for mixed differences}

\section{Introduction}
Given $G$ a group, the function $f:G\to\mathbb{C}$ is called a polynomial of degree $\leq n$ if  satisfies Fréchet's  mixed differences equation 
\begin{equation} \label{mix}
\Delta_{h_{n+1}}\Delta_{h_n}\cdots \Delta_{h_{1}}f=0 \text{ for all } h_1,\cdots, h_{n+1}\in G,
\end{equation}
and it is called a semipolynomial of degree  $\leq n$ if satisfies Fréchet's  unmixed differences equation 
\begin{equation} \label{unmix}
\Delta_{h}^{n+1} f=0 \text{ for all } h\in G.
\end{equation}
Here, $\Delta_h:\mathbb{C}^G\to \mathbb{C}^G$ is the forward differences operator given by $\Delta_hf(x)=f(xh)-f(x)$, $\Delta_{h_{n+1}}\Delta_{h_n}\cdots \Delta_{h_{1}} = \Delta_{h_{n+1}}(\Delta_{h_n}\cdots \Delta_{h_{1}})$ denotes the composition of the operators  
$\Delta_{h_1}, \cdots, \Delta_{h_{n+1}}$  and $\Delta_{h}^{n+1} = \Delta_{h} (\Delta_{h}^{n})$. If $G$ is commutative,  Djokovi\'{c}'s Theorem \cite{Dj} guarantees that equations \eqref{mix} and \eqref{unmix} are equivalent (see also \cite{L_Dj} for a different proof of this result) so that both concepts represent the very same class of functions.  
Finally, we say that $f:G\to\mathbb{C}$ is an exponential polynomial if the space $\tau(f)=\mathbf{span}\{\tau_h(f):h\in G\}$ is finite dimensional. Here $\tau_h:\mathbb{C}^G\to \mathbb{C}^G$ is the translation operator defined by $\tau_hf(x)=f(xh)$. Obviously, a function $f$ is an exponential polynomial if and only if belongs to some finite-dimensional translation invariant vector subspace $V$ of $\mathbb{C}^G$. A main motivation for this definition is that, for $G=\mathbb{R}^d$, a continuous function $f:\mathbb{R}^d\to\mathbb{C}$ is an exponential polynomial if and only if is a finite sum of functions of the form $p(x)e^{\langle x,\lambda\rangle}$ for certain ordinary polynomials $p(x)$ and vectors  $\lambda\in\mathbb{C}^d$. Moreover, if $f$ is a complex valued Schwartz distribution which belongs to a finite-dimensional translation invariant space of distributions, then $f$ is equal, in distributional sense, to a continuous complex valued exponential polynomial on $\mathbb{R}^d$ (see, e.g., \cite{anselone}, \cite{E}, \cite{Laird1}, \cite{Lo} or \cite{St} for the proofs of these claims) 

The following result, which was proved in \cite{AS_AM} (see also  \cite{AA_AM}, \cite{AK_CJM}), generalizes a well known theorem of Montel \cite{montel_1935}, \cite{montel}:

\begin{theorem} \label{MT_unmixed}
Let $G$ be a commutative (topological) group and $f:G\to \mathbb{C}$ be a (continuous) function. Assume that $E=\{h_1,\cdots,h_s\}$ (topologically) generates $G$ and let $\Delta_{h_k}^{n_k+1}f=0$, for $k=1,\cdots, s$. Then $f$  is a polynomial on $G$ of degree at most $n_1+\cdots+n_s$.
Moreover, if $G=\mathbb{R}^d$, $\{h_1,\cdots,h_s\}$ topologically generate $\mathbb{R}^d$ and $f$ is a complex valued Schwartz distribution on $\mathbb{R}^d$ such that  $\Delta_{h_k}^{n_k+1}f=0$, for $k=1,\cdots, s$. Then $f$  is, in distributional sense, a continuous  polynomial on $\mathbb{R}^d$ of degree at most $n_1+\cdots+n_s$. In particular, $f$ is equal almost everywhere to an ordinary polynomial with total degree at most $n_1+\cdots+n_s$.
\end{theorem}


In 1948 Montel \cite{montel_mixed1}, \cite{montel_mixed2} demonstrated, for mixed differences, a version of his theorem,  for complex functions of one and two real variables.  We generalize his result, with an easier proof, to complex functions defined on any finitely generated (topologically finitely generated) group (topological group) $G$, and to complex valued Schwartz distributions defined on $\mathbb{R}^d$. In particular, our result applies to complex valued functions depending on any finite number of real variables and serves for the characterization of polynomials and exponential polynomials as solutions of certain finite sets of mixed-differences functional equations. 


\section{Main results}

\begin{theorem} \label{main} 
Let  $G$ be a commutative group and let $f:G\to \mathbb{C}$  be a function. If  there exist exponential polynomials $P_k:G\to \mathbb{C}$, $k=1,\cdots, s$ and elements $\{h_{i,j}\}_{1\leq i\leq n;1\leq j\leq s}$ of $G$ such that,  for every $(i_1,\cdots,i_s)\in\{1,2,\cdots,n\}^s$, the set $\{h_{i_1,1},\cdots,h_{i_s,s}\}$ is a generating system of $G$, and
 \begin{equation}\label{teomontelmixtafuerte}
\Delta_{h_{1,k},h_{2,k},\cdots,h_{n,k}}f(x)=P_k(x) \text{, for all  }   k=1,2,\cdots,s, \text{ and all } x\in G,
\end{equation}
then $f$ is an exponential polynomial.
\end{theorem}

\begin{proof}  We proceed by induction on $n$. 
If we assume that 
$\Delta_{h_{1,k}}f=P_k$, $k=1,\cdots, s$, then we have that $\Delta_{h_{1,k}}(V)\subseteq V$ for $k=1,\cdots,s$, with $V=\mathbf{span}\{f\}+\tau(P_1)+\tau(P_2)+\cdots +\tau(P_s)$ 
Hence $V$ is a finite dimensional translation invariant space, which implies that all its elements are exponential polynomials, and $f\in V$. This proves the result for  $n=1$.

Assume the result holds true for $n-1$ and let $f$ satisfy \eqref{teomontelmixtafuerte}. Take $i=(i_1,\cdots,i_s) \in\{1,2,\cdots,n\}^s$ and define the function $F_i=\Delta_{h_{i_1,1}, h_{i_2,2},\cdots,h_{i_s,s}}f$. Then, if we denote by $\widehat{h_{i_k,k}}$ the fact that the step $h_{i_k,k}$ is hidden (i.e. it does not appear) in a formula, we have that
\begin{eqnarray*}
\Delta_{h_{1,k},h_{2,k},\cdots,\widehat{h_{i_k,k}}, \cdots, h_{n,k}}(F_i) &=& \Delta_{h_{1,k},h_{2,k},\cdots,\widehat{h_{i_k,k}}, \cdots, h_{n,k}}(\Delta_{h_{i_1,1}, h_{i_2,2},\cdots,h_{i_s,s}}f)\\
&=& \Delta_{h_{i_1,1}, h_{i_2,2}, \cdots,\widehat{h_{i_k,k}},\cdots,h_{i_s,s}} (\Delta_{h_{1,k},h_{2,k},\cdots,h_{i_k,k}, \cdots, h_{n,k}}f) \\
&=& \Delta_{h_{i_1,1}, h_{i_2,2}, \cdots,\widehat{h_{i_k,k}},\cdots,h_{i_s,s}} (P_k)=Q_k\\
\end{eqnarray*} 
with $Q_k$ being an exponential polynomial  for $k=1,2\cdots, s$. Hence the induction hypothesis tell us that $F_i$ is an exponential polynomial   for every $i$.

We want to reduce the size of the operator $\Delta_{h_{i_1,1}, h_{i_2,2},\cdots,h_{i_s,s}}$ used for the definition of $F_i$. So, consider, for each $k\in\{1,\cdots,s\}$,  the new function $G_{i,k}= \Delta_{h_{i_1,1},\cdots,\widehat{h_{i_k,k}},\cdots, h_{i_s,s}}f$ (which results from deleting the step $h_{i_k,k}$ in the definition of $F_i$). Then we have that
\begin{eqnarray*}
\Delta_{h_{1,k},h_{2,k},\cdots,h_{n-1,k}}(G_{i,k}) &=& \Delta_{h_{1,k},h_{2,k},\cdots,h_{n-1,k}}(\Delta_{h_{i_1,1},\cdots,\widehat{h_{i_k,k}},\cdots, h_{i_s,s}}f)\\
 &=& \Delta_{h_{1,k},\cdots \widehat{h_{k,k}},\cdots,h_{n-1,k}}(\Delta_{h_{i_1,1},\cdots, h_{i_{k-1},k-1},h_{k,k},  h_{i_{k+1},k+1},\cdots,h_{i_s,s}}f)
\end{eqnarray*}
is an exponential polynomial , since $\Delta_{h_{i_1,1},\cdots, h_{i_{k-1},k-1},h_{k,k},  h_{i_{k+1},k+1},\cdots,h_{i_s,s}}f$ is an exponential polynomial. Furthermore, for $j\neq k$ we have that
\begin{eqnarray*}
\Delta_{h_{1,j},h_{2,j},\cdots,,\widehat{h_{i_j,j}},\cdots, h_{n,j}}(G_{i,k}) &=& \Delta_{h_{1,j},h_{2,j},\cdots,,\widehat{h_{i_j,j}},\cdots, h_{n,k}}(\Delta_{h_{i_1,1},\cdots,\widehat{h_{i_k,k}},\cdots, h_{i_s,s}}f)\\
 &=& \Delta_{h_{i_1,1},,\cdots,\widehat{h_{i_j},j},\cdots,\widehat{h_{i_k,k}},\cdots, h_{i_s,s}}(\Delta_{h_{1,j},h_{2,j},\cdots,h_{i_j-1,j},h_{i_j,j}, h_{i_j+1,j}\cdots, h_{n,j}}f)\\
  &=& \Delta_{h_{i_1,1},\cdots,\widehat{h_{i_j},j},\cdots,\widehat{h_{i_k,k}},\cdots, h_{i_s,s}}(P_j)\\
\end{eqnarray*}
which is also an exponential polynomial. Hence the function $G_{i,k}$ satisfies a set of equations like \eqref{teomontelmixtafuerte} with $n-1$ steps, and the induction hypothesis implies that $G_{i,k}$ is also an exponential polynomial. In other words,  every function of the form $\Delta_{h_{j_1,a_1}, h_{j_2,a_2}, \cdots, h_{j_{s-1},a_{s-1}}}f$ with $0\leq j_k\leq n$ and $\{a_1,\cdots,a_{s-1}\}$ any subset of cardinality $s-1$ of $\{1,\cdots,s\}$, is an exponential polynomial .

Let us still do a step more: Take $t\in\{1,\cdots,s\}$, $t\neq k$, and consider the function $G_{i,k,t}=\Delta_{h_{i_1,1},\cdots,\widehat{h_{i_k,k}},\cdots,\widehat{h_{i_t,t}},\cdots,h_{i_s,s}}f$.  Then
\begin{eqnarray*}
\Delta_{h_{1,k},h_{2,k},\cdots,h_{n-1,k}}(G_{i,k,t}) &=& \Delta_{h_{1,k},h_{2,k},\cdots,h_{n-1,k}}(\Delta_{h_{i_1,1},\cdots,\widehat{h_{i_k,k}},\cdots, \widehat{h_{i_t,t}},\cdots,h_{i_s,s}}f)\\
 &=& \Delta_{h_{1,k},\cdots \widehat{h_{k,k}},\cdots,h_{n-1,k}}(\Delta_{h_{i_1,1},\cdots, h_{i_{k-1},k-1},h_{k,k},  h_{i_{k+1},k+1},\cdots, \widehat{h_{i_t,t}},\cdots, h_{i_s,s}}f)
\end{eqnarray*}
is an exponential polynomial , since $\Delta_{h_{i_1,1},\cdots, h_{i_{k-1},k-1},h_{k,k},  h_{i_{k+1},k+1},\cdots,  \widehat{h_{i_t,t}},\cdots, h_{i_s,s}}f$ is an exponential polynomial. Furthermore,
\begin{eqnarray*}
\Delta_{h_{1,t},h_{2,t},\cdots,h_{n-1,t}}(G_{i,k,t}) &=& \Delta_{h_{1,t},h_{2,t},\cdots,h_{n-1,t}}(\Delta_{h_{i_1,1},\cdots,\widehat{h_{i_k,k}},\cdots, \widehat{h_{i_t,t}},\cdots,h_{i_s,s}}f)\\
 &=& \Delta_{h_{1,t},\cdots \widehat{h_{t,t}},\cdots,h_{n-1,t}}(\Delta_{h_{i_1,1},\cdots,  \widehat{h_{i_k,k}},\cdots h_{i_{t-1},t-1},h_{t,t},  h_{i_{t+1},t+1},\cdots,\cdots, h_{i_s,s}}f)
\end{eqnarray*}
is an exponential polynomial , since $\Delta_{h_{i_1,1},\cdots,  \widehat{h_{i_k,k}},\cdots h_{i_{t-1},t-1},h_{t,t},  h_{i_{t+1},t+1},\cdots,\cdots, h_{i_s,s}}f$ is an exponential polynomial .
Finally, for $j\in\{1,\cdots,s\}\setminus \{k,t\}$ we have that
\begin{eqnarray*}
&\ & \Delta_{h_{1,j},h_{2,j},\cdots,,\widehat{h_{i_j,j}},\cdots, h_{n,j}}(G_{i,k,t}) = \Delta_{h_{1,j},h_{2,j},\cdots,,\widehat{h_{i_j,j}},\cdots, h_{n,k}}(\Delta_{h_{i_1,1},\cdots,\widehat{h_{i_k,k}},\cdots,\widehat{h_{i_t,t}},\cdots,h_{i_s,s}}f)\\
 &\ & =  \Delta_{h_{i_1,1},,\cdots,\widehat{h_{i_j},j},\cdots,\widehat{h_{i_k,k}},\cdots, \widehat{h_{i_t,t}},\cdots, h_{i_s,s}}(\Delta_{h_{1,j},h_{2,j},\cdots,h_{i_j-1,j},h_{i_j,j}, h_{i_j+1,j}\cdots, h_{n,j}}f)\\
  &\ & =  \Delta_{h_{i_1,1},,\cdots,\widehat{h_{i_j},j},\cdots,\widehat{h_{i_k,k}},\cdots, \widehat{h_{i_t,t}},\cdots, h_{i_s,s}}(P_j)\\
\end{eqnarray*}
which is also an exponential polynomial. Hence the function $G_{i,k,t}$ satisfies a set of equations like \eqref{teomontelmixtafuerte} with $n-1$ steps, and the induction hypothesis implies that $G_{i,k,t}$ is also an exponential polynomial. Thus,  every function of the form $\Delta_{h_{j_1,a_1}, h_{j_2,a_2}, \cdots, h_{j_{s-2},a_{s-2}}}f$ with $0\leq j_k\leq n$ and $\{a_1,\cdots,a_{s-2}\}$ any subset of cardinality $s-2$ of $\{1,\cdots,s\}$, is an exponential polynomial .

We can repeat the argument to delete  another step in the difference operators used for the definition of each one of the functions $G_{i,k,t}$ above, and maintain the property that the new functions are still exponential polynomials. Iterating the argument as many times as necessary we lead to the fact that  all  functions
$\Delta_{h_{i,j}}f$ are exponential polynomials. Then we apply the case $n=1$ to conclude that $f$  is an exponential polynomial. This ends the proof.

\end{proof}
The following are easy corollaries of Theorem \ref{main}:
\begin{corollary} \label{co1}
Let  $G$ be a topological commutative group and let $f:G\to \mathbb{C}$  be a continuous function. If  there exist exponential polynomials $P_k:G\to \mathbb{C}$, $k=1,\cdots, s$ and elements $\{h_{i,j}\}_{1\leq i\leq n;1\leq j\leq s}$ of $G$ such that,  for every $(i_1,\cdots,i_s)\in\{1,2,\cdots,n\}^s$, the set $\{h_{i_1,1},\cdots,h_{i_s,s}\}$ topologically generates $G$, and $f$ satisfies \eqref{teomontelmixtafuerte}, then $f$ is an exponential polynomial.
\end{corollary}

\begin{corollary} \label{co2}
 If $\{h_{i_1,1},\cdots,h_{i_s,s}\}$ topologically generates $\mathbb{R}^d$ for every $(i_1,\cdots,i_s)\in\{1,\cdots,n\}^s$, and $f\in\mathcal{D}(\mathbb{R}^d)'$ is a complex valued Schwartz distribution on $\mathbb{R}^d$ which satisfies the equations \eqref{teomontelmixtafuerte} for certain continuous exponential polynomials $P_k$, then $f$ is, in distributional sense, a continuous exponential polynomial. In particular,  there exists a continuous exponential polynomial $p$ such that $f=p$ almost everywhere.
\end{corollary}

\begin{proof}
It is enough to follow the very same steps of the demonstration of Theorem \ref{main}, just taking into account Anselone-Korevaar's theorem and that the operators $\tau_h$ and $\Delta_h$, which are defined for Schwartz distributions $f\in \mathcal{D}(\mathbb{R}^d)'$ by the expressions 
$\langle \tau_h(f),\varphi\rangle = \langle f,\tau_{-h}(\varphi)\rangle$ and  $\langle \Delta_h(f),\varphi\rangle = \langle f,\Delta_{-h}(\varphi)\rangle$, respectively, inherit all properties from their corresponding versions, originally defined for ordinary functions.
\end{proof}
\begin{remark}
If we impose $P_k=0$ for all $k$ in Theorem \ref{main} or Corollaries \ref{co1}, \ref{co2}, then the function $f$ will be a polynomial. Thus, these results generalize Theorem \ref{MT_unmixed} to the mixed differences case.
\end{remark}

\begin{remark} \label{notita}
The condition that, for every $(i_1,\cdots,i_s)\in\{1,2,\cdots,n\}^s$,   $\{h_{i_1,1},\cdots,h_{i_s,s}\}$ either generates or (in the case that $G$ is a topological group) topologically generates $G$,  is a very natural necessary condition. For example, for $G=\mathbb{R}^d$, if exists $(i_1,\cdots,i_s)\in\{1,2,\cdots,n\}^s$ such that  $h_{i_1,1}\mathbb{Z}+h_{i_2,2}\mathbb{Z}+\cdots+h_{i_s,s}\mathbb{Z}$ is not dense in $\mathbb{R}^d$, then there exists a non-differentiable continuous function $f:\mathbb{R}^d\to\mathbb{R}$ such that $\Delta_{h_{i_k,k}}f=0$ for $k=1,\cdots, s$ and, henceforth, $f$ solves the functional equations \eqref{teomontelmixtafuerte} with $P_k=0$ for all $k$ (see \cite{A_An-Kor-Th}  for a proof of this claim) but is not a polynomial nor an exponential polynomial, since every continuous polynomial on $\mathbb{R}^d$ is an ordinary polynomial in $d$ variables \cite{frechet} and, henceforth, is a  differentiable function, and continuous exponential polynomials on $\mathbb{R}^d$ are also differentiable functions.  The merit of Theorem \ref{main} and Corollaries \ref{co1}, \ref{co2}  is, thus, to prove that their hypotheses  are sufficient  to guarantee that $f$ is an exponential polynomial polynomial.
\end{remark}

\bigskip 

\footnotesize{Jose Maria Almira

Departamento de Ingenier\'{\i}a y Tecnolog\'{\i}a de Computadores, Facultad de Inform\'{a}tica

Universidad de Murcia

Campus de Espinardo

30100 Murcia (Spain)

\vspace{2mm}

and

\vspace{2mm}

Departamento de Matem\'{a}ticas, Universidad de Ja\'{e}n

E.P.S. Linares,  C/Alfonso X el Sabio, 28

23700 Linares (Ja\'{e}n) Spain

e-mail: jmalmira@um.es; jmalmira@ujaen.es 

}

\end{document}